\documentclass[11pt]{amsart}
\usepackage{amssymb,amsmath,amsthm,amsfonts,amsopn,url,color,enumerate,microtype,stmaryrd,tikz-cd}
\usepackage[normalem]{ulem}
\usepackage[all]{xy}
\usepackage{amscd}
\usepackage{thmtools, thm-restate}
\usepackage{listings}
\usepackage{xcolor}
\usepackage{multicol}
\usepackage{hyperref}
\theoremstyle{plain}
\newtheorem{theorem}{Theorem}[section]
\newtheorem{prop}[theorem]{Proposition}
\newtheorem{lemma}[theorem]{Lemma}
\newtheorem{corollary}[theorem]{Corollary}
\theoremstyle{definition}
\newtheorem{defn}[theorem]{Definition}
\newtheorem{rem}[theorem]{Remark}
\newtheorem{notation}[theorem]{Notation}

\newenvironment{claim}[1]{\par\noindent\underline{Claim:}\space#1}{}
\newenvironment{claimproof}[1]{\par\noindent\underline{Proof:}\space#1}{\hfill $\blacksquare$}

\newcommand{\CM}{Cohen-Macaulay}
\newcommand{\field}[1]{\mathbb{#1}}
\newcommand{\N}{\field{N}}

\newcommand{\func}[1]{\mathrm{#1} \,}

\newcommand{\coker}{\func{coker}}
\newcommand{\im}{\func{im}}

\newcommand{\arrow}[1]{\stackrel{#1}{\rightarrow}}

\DeclareMathOperator{\cok}{coker}

\DeclareMathOperator{\ann}{ann}

\DeclareMathOperator{\Hom}{Hom}

\newcommand{\be}{\begin{enumerate}}
\newcommand{\ee}{\end{enumerate}}

\newcommand{\li}%{{\mathrm{lic}}}
 {\leftfootline}

\newcommand{\cM}{\mathcal{M}}

\renewcommand{\phi}{\varphi}

\newenvironment{psmallmatrix}
  {\left(\begin{smallmatrix}}
  {\end{smallmatrix}\right)}

\newcommand{\cl}{{\mathrm{cl}}}
\let\int\relax
\DeclareMathOperator{\int}{i}

\newcommand{\nzd}{non-zerodivisor}
\newcommand{\fg}{finitely generated}

\DeclareMathOperator{\tr}{tr}
\DeclareMathOperator{\id}{Id}
\newcommand{\MCM}{maximal \CM}
\newcommand{\charp}{characteristic $p>0$}

\DeclareMathOperator{\textim}{im}

\title[Test ideals over rings of finite and countable CM type]{Cohen-Macaulay test ideals over
rings of finite and countable Cohen-Macaulay type}
\date{\today}
\begin{document}

\author[J. Benali, S. Pothagoni, R. R.G.]{Julian Benali, Shrunal Pothagoni, and Rebecca R.G.}

\begin{abstract}

The third named author and P\'{e}rez proved that under certain conditions the test ideal of a module closure agrees with the trace ideal of the module closure.
We use this fact to compute the test ideals of various rings with respect to the closures coming from their indecomposable \MCM\ modules.  We also give an easier way to compute the test ideal of a hypersurface ring in 3 variables coming from a module with a particular type of matrix factorization.

\end{abstract}

\maketitle

\tableofcontents

\section{Introduction}

The test ideal of a noetherian local ring is an invariant that helps us determine how singular the ring is. It was originally defined in terms of tight closure for rings of \charp\ \cite{tightclosure,HH94,testidealssurvey}. Later it was proved that the test ideal agrees with a characteristic 0 invariant from algebraic geometry, the multiplier ideal, via reduction to \charp\ \cite{Smith00,Hara01}. The third named author studied test ideals coming from arbitrary closure operations in \cite{PRC}, and proved that the test ideals coming from big \CM\ module closures give information on the singularities of the ring.

One of the main results of \cite{PRC} is that the test ideal of a module closure agrees with the trace ideal of the module when the ring is local or the module is finitely-presented. This enables us to study the singularities of the ring by computing trace ideals. In this paper, we compute trace ideals of \fg\ indecomposable maximal \CM\ modules over rings of finite or countable \CM\ type, providing examples for some of the results of \cite{PRC}.

The following result inspired the computations in this paper:

\begin{prop}[{\cite[Proposition 4.13]{PRC}}]
Let $(R,m)$ be a \CM\ ring with finite \CM\ type. If $R$ is not regular, then 
\[\sqrt{\tau_{MCM}(R)}=m.\]
\end{prop}

This leads to two questions:
\begin{enumerate}
    \item Which examples give $\tau_{MCM}(R)=m$?
    \item When $R$ does not have finite \CM\ type, does this still hold?
\end{enumerate}

The second question was answered negatively in \cite{PRC}, in the case of the Whitney Umbrella. In this paper, we give many more examples, fleshing out the details of this result.

Another question in \cite{PRC} is whether $\tau_{MCM}(R)$ is always nonzero. Proposition 4.18 of that paper indicates that $\tau_{MCM}(R)$ is nonzero when $R$ is a local domain of dimension 1. In this paper, we give examples where $R$ is not a domain and $\tau_{MCM}(R)=0$ (see Proposition \ref{prop:zero}).

In Section \ref{sec:veronese}, we compute trace ideals of \fg\ \CM\ modules over Veronese subrings via direct computation. In Section \ref{sec:maintheorem}, we prove a result on trace ideals of modules whose resolutions take the form of matrix factorizations. This is a special case of the result stated by Faber in \cite[Section 2.3]{Faber}, based on
previous work of Vasconcelos \cite[Proposition 6.2.3]{Vasconcelos}.
In Sections \ref{sec:dim1ade} and \ref{sec:dim2ade}, we apply this result to compute trace ideals of \fg\ \CM\ modules over dimension 1 and 2 ADE singularities, including some that are not domains or have countable \CM\ type.

\section{Background/Previous Results}

\begin{rem}All rings are assumed to be commutative, unital, Noetherian, and local. We will use $k$ to denote a field, $R$ to denote a ring, and $M$ to denote an $R$-module. The notation $(R,m,k)$ indicates that $R$ is a local ring with maximal ideal $m$ and residual field $k=R/m$.
\end{rem}

\begin{defn}The \textit{Krull dimension} of $R$ is the length of the longest proper chain of prime ideals $p_0\subsetneq p_1\subsetneq\cdots\subsetneq p_n$ of $R$. \end{defn}

\begin{defn} A sequence $x_1,...,x_n$ of elements in $R$ is said to be a \textit{regular sequence} on $M$ if $(x_1,...,x_n)M\neq M$, $x_1$ is a \nzd\ on $M$, and $x_i$ is a \nzd\ on $M/(x_1,...,x_{i-1})M$ for all $i\in\{2,...,n\}$. \end{defn}

\begin{defn} We say $M$ is a \textit{\CM\ module} over $R$ if the length of the longest regular sequence on $M$ is the same as the Krull dimension of $R$. When $M$ is \fg, it is often called a \textit{\MCM\ module}, and we will use this term throughout the paper. \end{defn}

\begin{rem}
The reason for the name \MCM\ module is that sometimes \CM\ modules are defined to be modules $M$ for which the length of the longest regular sequence on $M$ is the same as the Krull dimenson of $M$. In this case the \CM\ modules whose Krull dimension is maximal (i.e. equal to $\dim R$) are called maximal \CM\ modules.
\end{rem}

Next we give the definition of the test ideal, and describe its relationship to the trace ideal.

\begin{defn}[{\cite[Definition 2.3]{rrg1}}]
Let $M$ be an $R$-module. We define the \textit{module closure $\cl_M$} to be the operation that takes an $R$-module $A$ contained in another $R$-module $B$ to the $R$-module
\[ A_B^{\cl_M}:=\{b \in B : x \otimes b \in \im(M \otimes_R A \to M \otimes_R B) \text{ for all } x \in M\}.
\]
By \cite{rrg1} this is a closure operation, i.e.:
\begin{enumerate}
    \item $A \subseteq A_B^{\cl_M} \subseteq B$,
    \item $(A_B^{\cl_M})_B^{\cl_M}=A_B^{\cl_M},$
    \item and if $A \subseteq C \subseteq B$, then $A_B^{\cl_M} \subseteq C_B^{\cl_M}.$
\end{enumerate}
\end{defn}

\begin{defn}[{\cite[Definition 4.7]{bigcmalgebraaxiom}}]
Let $M$ be an $R$-module. The \textit{test ideal of $R$ associated to $M$} is the ideal
\[\tau_M(R):=\bigcap_{\substack{A \subseteq B\\  \textit{ $R$-modules}}} (A:_R A_B^{\cl_M}.)\]
\end{defn}

\begin{defn} If $R$ is a complete local domain and $M$ an $R$-module, the \textit{trace ideal of $M$} is $$\tr_{M}(R)=\sum_{f\in\Hom(M,R)} f(M). $$ That is, we find the images of the $R$-module homomorphisms from $M$ to $R$ and take their sum.\end{defn}

\begin{theorem}[{\cite[Theorem 3.12]{PRC}}]
Let $R$ be local and $M$ be an $R$-module. If $R$ is complete or $M$ is \fg\, then $\tau_M(R)=\tr_M(R)$.
\end{theorem}

We will use $\tau_M(R)$ to denote this ideal in the rest of the paper.

This implies that to compute test ideals of \MCM\ $R$-modules, we can compute the trace ideals of these modules, which is a much more doable task. We will in fact go a step further and compute the \MCM\ module test ideal of $R$, defined as follows:

\begin{defn}[{\cite[Section 4]{PRC}}]
Let $MCM(R)$ denote the set of all isomorphism classes of \MCM\ $R$-modules, i.e. \fg\ \CM\ $R$-modules.
The \textit{\MCM\ module test ideal} of $R$ is
\[\tau_{MCM}(R):=\bigcap_{M \in MCM(R)} \tau_M(R).\]
\end{defn}

\begin{rem}
For further discussion of how to reduce classes of \CM\ modules to sets, see \cite[Remark 4.6]{PRC}.
\end{rem}

\begin{rem}\label{rem:goal} Our goal is to compute $\tau_{MCM}(R)$. In practice, we do not need to compute the test ideal of every \MCM\ $R$-module in order to compute $\tau_{MCM}(R)$. The following are two shortcuts that make this task more doable:
\begin{enumerate}
    \item Let $M$ be a nonzero free $R$-module. Then any projection map from $M$ to $R$ is a surjective $R$-module homomorphism and thus, $\tau_M(R)=R$.
    \item Every \MCM\ module is a direct sum of indecomposable \MCM\ modules. Furthermore, it follows from the definition that for any \CM\ $R$-modules, $N$ and $L$, $\tau_{N\oplus L}(R)=\tau_N(R)+\tau_L(R)$. Hence to compute $\tau_{MCM}(R)$, it suffices to compute the intersection over the indecomposable \MCM\ $R$-modules.
\end{enumerate}  
From these facts, we conclude that $\tau_{MCM}(R)$ is the intersection of the test ideals of the non-free indecomposable MCM $R$-modules. 
\end{rem}

\section{First Example: Veronese subrings}
\label{sec:veronese}

 In this section, we compute the test ideal of the (complete versions of the) Veronese subrings of the polynomial ring in two variables. This example is explained in detail to give a framework for how these computations are done.

\begin{prop} 
If $R=k\llbracket  x^d,x^{d-1}y,...,xy^{d-1},y^d\rrbracket$ where $k$ is a field, then $\tau_{MCM}(R)=m$. Furthermore, $\tau_M(R)=m$ for all non-free indecomposable \MCM\ modules $M$.
\end{prop}

\begin{proof}
Up to isomorphism, the set of non-free indecomposable \MCM\ 
$R$-modules is \[\cM=\{R, xR+yR,x^2R+xyR+y^2R,\ldots,x^{d-1}R+x^{d-2}yR+\cdots+xy^{d-2}R+y^{d-1}R\}.\] This follows from \cite[Corollary 6.4]{leuschkewiegand}, which states that the indecomposable \MCM\ $R$-modules are isomorphic to the direct summands of $S=k\llbracket x,y \rrbracket$ as an $R$-module.
We compute the trace ideals $\tau_{M}(R)$ for $M\in\cM$ and then compute their intersection $\tau_{MCM}(R).$ We claim that $\tau_{R}(R)=R$ and if $M\in\mathcal{M}\setminus\{R\}$, then $\tau_{M}(R)=m$ where $m$ is the maximal ideal.

Let $S=k\llbracket x,y\rrbracket$. We will view $R$ as a subring of $S$. Since $\Hom_R(R,R)$ contains the identity map, $\tau_R(R)=R$. This will be true for any ring as previously stated in Remark \ref{rem:goal}. 

Now let
\[M_n=x^nR+x^{n-1}yR+\cdots+xy^{n-1}R+y^nR=\sum_{i=0}^nx^{n-i}y^iR\in\mathcal{M},\]
where $1 \le n \le d-1$.

\begin{claim} $\Hom(M_n,R)$ is generated by the maps $f_i$ for $i\in\{0,...,d-n\}$ defined by $f_i(p)=x^{d-n-i}y^ip$ for all $p\in M_n$.\\
\end{claim}

\begin{claimproof}
Let $f\in\Hom(M_n,R)$. For each $j\in\{0,...,n\}$, we see that
\[\begin{aligned}
x^jy^{d-j}f(x^{n-j}y^j)&=f(x^ny^d)=x^ny^{d-n}f(y^n)\\
x^{d-n+j}y^{n-j}f(x^{n-j}y^j)&=f(x^dy^n)=x^{d-n}y^nf(x^n).
\end{aligned}\]
Since $x$ and $y$ are \nzd s in $S$, this implies
\[\begin{aligned}
y^{d-j}f(x^{n-j}y^j)&=x^{n-j}y^{d-n}f(y^n)\in (x^{n-j})S\\
x^{d-n+j}f(x^{n-j}y^j)&=y^jf(x^n)\in (y^j)S.
\end{aligned}\]
As $x^{n-j},y^{d-j}$ and $y^j,x^{d-n+j}$ are both regular sequences in $S$, we now have
\[f(x^{n-j}y^j)\in (x^{n-j})S\cap (y^j)S=(x^{n-j}y^j)S.\]

Finally, we obtain 
\[f(x^{n-j}y^j)\in (x^{n-j}y^j)S\cap R=(x^{d-j}y^j,x^{d-j-1}y^{j+1},...,x^{n-j}y^{d-n+j})R.\]

Now, since $f(x^n)\in(x^d,x^{d-1}y,...,x^ny^{d-n})R$ there exist $p_0,...,p_{d-n}\in R$ such that
\[\begin{aligned}f(x^n)&=p_0x^d+p_1x^{d-1}y+\cdots+p_{d-n}x^ny^{d-n}\\
&=p_0f_0(x^n)+p_1f_1(x^n)+\cdots+p_{d-n}f_{d-n}(x^n)\\  &=\sum_{i=0}^{d-n}p_if_i(x^n)\end{aligned}\] Then for each $j\in\{0,...,n\}$, we have
\[
\begin{aligned}
x^jy^{d-j}f(x^{n-j}y^j)&=f(x^ny^d)=y^df(x^n) \\
&=y^d\sum_{i=0}^{d-n}p_ix^{d-i}y^i \\
&=x^jy^{d-j}\sum_{i=0}^{d-n}p_ix^{d-i-j}y^{i+j} \\ &=x^jy^{d-j}\sum_{i=0}^{d-n}p_if_i(x^{n-j}y^j).
\end{aligned}\]
Since $x^jx^{d-j}$ is a \nzd\ in $R$, this yields $f(x^{n-j}y^j)=\sum_{i=0}^{d-n}p_if_i(x^{n-j}y^j)$. Thus $f=\sum_{i=0}^{d-n}p_if_i$, so that $f_0,...,f_{d-n}$ generate $\Hom(M_n,R)$, as claimed.
\end{claimproof}

We can see that $f_i(M_n)=(x^{d-i}y^i,x^{d-i-1}y^{i+1},...,x^{d-n-i}y^{n+i})$ for each $i$ so that \[
\begin{aligned}
\tau_{M_n}(R) &=\sum_{i=0}^{d-n}f_i(M_n)=\sum_{i=0}^{d-n}(x^{d-i}y^i,x^{d-i-1}y^{i+1},...,x^{d-n-i}y^{n+i}) \\ &= (x^d,x^{d-1}y,...,y^d)=m \\
\end{aligned}
\]
as desired.

Therefore, we conclude that $\tau_{MCM}(R)=m$.
\end{proof}

\section{Test Ideals of Modules with Matrix Factorizations}
\label{sec:maintheorem}

Although it is possible to manually compute the test ideal in many more cases, it will become apparent in later sections that this approach is not always optimal. The results of this section allows us to easily compute the test ideals of any module of the form stated in Corollary \ref{corollary:factorization}. This will reduce the length and complexity of the computations for examples like the Kleinian singularities.

Theorem \ref{thm:maintheorem} and Corollary \ref{corollary:factorization} are a special case of the result stated by Faber in \cite{Faber}, based on Proposition 6.2.3 of \cite{Vasconcelos}. We give a detailed proof in the special case that the ring is of the form $R=S\llbracket z\rrbracket/(z^2+g)$ and the module has a matrix factorization, relying as much as possible on linear algebra and properties of $\Hom$.

\begin{lemma}
\label{lemma:kertrival}
Let $S$ be a Noetherian ring and let $\varphi$ be an $n\times n$ matrix over $S\llbracket z\rrbracket$ whose entries are all in $S$. Then $\ker(z\id_n-\varphi)$ is trivial.
\end{lemma}

\proof Let $R=S[z]$ and $\widehat{R}=S\llbracket z\rrbracket$, the $(z)$-adic completion of $R$. Let $\Phi=z\id_n-\varphi:(\widehat{R})^n\to (\widehat{R})^n$. 
Then $\Phi$ is the unique $\hat{R}$-linear extension of the $R$-linear map $\Psi:R^n\to R^n$ which, as a matrix, has the same entries as $\Phi$. Let $v \in \ker(\Psi)$ and assume $v \neq 0$. Let $v_i$ be a nonzero component of $v$ with maximal degree in $z$. We will use $\deg$ to refer to degree in $z$. Then we have that 
$$0=(\Psi v)_i=\sum_{j=1}^{n}\Psi_{ij}v_j=(z-\varphi_{ii})v_i+\sum_{\substack{j=1 \\ j \neq i}}^{n}\Psi_{ij}v_j.$$ 
We see that $\deg((z-\varphi_{ii})v_i))=\deg(v_i)+1$, but because the off-diagonal entries of $\Psi$ do not have any $z$'s in them,
$$\deg\left(\sum_{\substack{j=1 \\ j \neq i}}^{n}\Psi_{ij}v_j\right) \leq \max_{j\neq i}(\deg(\Psi_{ij}v_j))=$$ $$\max_{j\neq i}(\deg(v_j))\leq \deg(v_i).$$
This implies $(\Psi v)_i$ is a polynomial in $z$ of degree $\deg(v_i)+1$ and is therefore nonzero, a contradiction. Hence $\ker(\Psi)$ is trivial.

Now, because $S$, and therefore $R$, is Noetherian, $\widehat{R}$ is a flat $R$-module. Thus, since $\Psi$ is injective, the induced map $$\Psi\otimes id_{\widehat{R}}:R^n\otimes_{R}\widehat{R}\to R^n\otimes_{R}\widehat{R}$$
is also injective. However, $R^n\otimes_{R}\widehat{R}\cong (\widehat{R})^n$, so $\Psi\otimes id_{\widehat{R}}$ corresponds to the unique extension $\Phi$ of $\Psi$ to $(\widehat{R})^n$. Thus, $\Phi$ is injective as well.
\qed

\begin{defn}
Let $S$ be a ring and $x\in S$. A \textit{matrix factorization} for $x$ is a pair of maps of free $S$-modules ($\varphi: F \to G$, $\psi: G \to F$) such that $\varphi \psi = x \cdot 1_G$ and $\psi \varphi = x \cdot 1_F$. \cite{/eisen80}
\end{defn}

\begin{notation}
Let $S$ be a ring and let $R$ be a quotient ring of $S$. Then given a matrix $\Phi$ over $S$ we will use $\overline{\Phi}$ to denote the same matrix over the quotient $R$. 
\end{notation}

\begin{rem}
All matrix factorizations that appear in this paper are square. That is to say, all matrix factorizations which appear will be maps between free modules of the same rank. 
\end{rem}

\begin{lemma}
\label{lemma:kerimage}
Let $S$ be a Noetherian ring and $g\in S$. If $\varphi$ is an $n\times n$ matrix over $S\llbracket z\rrbracket$ whose entries lie in $S$ such that $(z\id_n-\varphi, z\id_n+\varphi)$ is a matrix factorization for $z^2+g$ in $S\llbracket z\rrbracket$ $($or equivalently, $\varphi^2=-g\id_n)$, then $\ker(\overline{z\id_n-\varphi})=\textim(\overline{z\id_n+\varphi})$ and $\ker(\overline{z\id_n-\varphi^\intercal})=\textim(\overline{z\id_n+\varphi^\intercal})$ in $R=S\llbracket z\rrbracket/(z^2+g)$.
\end{lemma}

\proof Since $(\overline{z\id_n-\varphi})(\overline{z\id_n+\varphi})=0$ over $R$, $\textim(\overline{z\id_n+\varphi})\subset\ker(\overline{z\id_n-\varphi})$. Let $\overline{v} \in \ker(\overline{z\id_n-\varphi})$ for some $v\in(S\llbracket x\rrbracket)^n$ so that $(\overline{z\id_n-\varphi})\overline{v}=0$. Then
\[(z\id_n-\varphi)v \in \textim((z^2+g)\id_n).\]
So, there exists $w\in(S\llbracket z\rrbracket)^n$ such that 
\[(z\id_n-\varphi)v=(z^2+g)w.\]
This can be rewritten as \[(z\id_n-\varphi)v=(z\id_n-\varphi)(z\id_n+\varphi)w.\] Since $\ker(z\id_n-\varphi)$ is trivial in $S\llbracket z\rrbracket$ by Lemma \ref{lemma:kertrival}, this implies \[v=(z\id_n+\varphi)w\in\textim(z\id_n+\varphi)\] and hence $\overline{v}\in\textim(\overline{z\id_n+\varphi})$ over $R$. Thus, $\textim(\overline{z\id_n+\varphi})=\ker(\overline{z\id_n-\varphi}).$

Now note that 

\begin{align*}
(z\id_n-\varphi^\intercal)(z\id_n+\varphi^\intercal)&=((z\id_n+\varphi)(z\id_n-\varphi))^\intercal\\
&=((z^2+g)\id_n)^\intercal\\
&=(z^2+g)\id_n
\end{align*}

so that $(z\id_n-\varphi^\intercal,z\id_n+\varphi^\intercal)$ is also a matrix factorization for $z^2+g$. So, by what was already shown, $\ker(\overline{z\id_n-\varphi^\intercal})=\textim(\overline{z\id_n+\varphi^\intercal})$. \qed

\begin{theorem} 
\label{thm:maintheorem}
Let $S$ be a Noetherian ring and $R=S\llbracket z\rrbracket/(z^2+g)$. Let $\varphi$ be an $n \times n$ matrix over $S\llbracket z\rrbracket$ whose entries lie in $S$ such that $(z\id_n-\varphi,z\id_n+\varphi)$ is a matrix factorization for $z^2+g$. If $M=\cok(\overline{z\id_n-\varphi})$, then $$\textim(\Hom_R(M,R)\hookrightarrow \Hom_R(R^n,R)\xrightarrow{\cong} R^n)=\textim(\overline{z\id_n+\varphi^\intercal)}.$$
\end{theorem} 

\begin{proof}

 Let $\Phi=\overline{z\id_n-\varphi}$. Consider the exact sequence

\begin{equation} 0\to\textim(\Phi)\xrightarrow{\iota} R^n\xrightarrow{p} M\to 0. \end{equation}

Since $\Hom_R(\cdot,R)$ is left exact, this yields an exact sequence

\begin{equation} \label{homseq} 0\to\Hom_R(M,R)\xrightarrow{p^*}\Hom_R(R^n,R)\xrightarrow{\iota^*}\Hom_R(\textim(\Phi),R). \end{equation}

Note that there is a standard $R$-module isomorphism $\alpha:R^n\to\Hom_R(R^n,R)$ which sends a vector $v$ to the dot product map $\langle v,\underline{\ \ }\rangle$. The map $\alpha^{-1}$ sends a map to the vector whose entries are the evaluations of the map at the standard basis vectors. That is, for any $f:R^n\to R$,

\[\alpha^{-1}(f)=\begin{psmallmatrix}
  f(e_1)\\
  f(e_2)\\
  \vdots \\
  \\
  f(e_n) \\
\end{psmallmatrix}\]

where $e_i$ is the $i$th standard basis vector. This isomorphism, combined with the exact sequence in (\ref{homseq}), yields a new exact sequence

\begin{equation} \label{oldstar}
0\to\Hom_R(M,R)\xrightarrow{\alpha^{-1}\circ p^*} R^n\xrightarrow{\iota^*\circ\alpha}\Hom_R(\textim(\Phi),R). 
\end{equation}

We claim that we have a commutative diagram

\[\begin{tikzcd}[cramped, sep=scriptsize]
R^n\arrow[r, "\alpha"]
    \arrow[rrrr, bend right, "\Phi^\intercal"] &
\Hom_R(R^n,R)\arrow[r, "\iota^*"] &
\Hom_R(\textim(\Phi),R)\arrow[r, "f\mapsto f\circ\Phi"] &
\Hom_R(R^n,R)\arrow[r, "\alpha^{-1}"] &
R^n
\end{tikzcd}.\]

First, $\alpha$ sends $v\in R^n$ to the dot product map $\langle v,\underline{\ \ }\rangle$. The second map, $\iota^*$, restricts this to $\textim(\Phi)$. The third map composes this with $\Phi$ to obtain $\langle v,\Phi\cdot\underline{\ \ }\rangle$. Finally, $\alpha^{-1}$ sends this to the vector \[\begin{psmallmatrix}
  \langle v,\Phi e_1\rangle\\ \langle v,\Phi e_2\rangle\\
  \vdots \\
  \\
  \langle v,\Phi e_n\rangle \\
\end{psmallmatrix}=\Phi^\intercal v.\]

Thus, the diagram commutes. Note that the third map is injective since it is the dual of the surjective map $R^n\xrightarrow{\Phi}\textim(\Phi)$ and $\Hom_R(\cdot,R)$ is left exact. Consequently, we have

\begin{equation} \label{equalitiesa}
\ker(\Phi^\intercal)=\ker(R^n\xrightarrow{\alpha}\Hom_R(R^n,R)\xrightarrow{\iota^*}\Hom_R(\textim(\Phi),R)). \end{equation}

Recall that, by exactness of (\ref{oldstar}), 

\begin{equation} \label{equalitiesb}
\begin{aligned}
&\ker(R^n\xrightarrow{\alpha}\Hom_R(R^n,R)\xrightarrow{\iota^*}\Hom_R(\textim(\Phi),R)) \\ &=\textim(\Hom_R(M,R)\xrightarrow{p^*}\Hom_R(R^n,R)\xrightarrow{\alpha^{-1}}R^n) \\
\end{aligned} \end{equation}

Finally, by Lemma \ref{lemma:kerimage}, \begin{equation} \label{equalitiesc} \textim(\overline{z\id_n+\varphi^\intercal})=\ker(\Phi^\intercal) \end{equation}

The sequence of equalities (\ref{equalitiesc}), (\ref{equalitiesa}), and (\ref{equalitiesb})  gives the desired result.

\end{proof}

\begin{corollary}
\label{corollary:factorization}
Let $S$ be a Noetherian ring and $R=S\llbracket z\rrbracket/(z^2+g)$. Let $\varphi$ be an $n \times n$ matrix over $S\llbracket z\rrbracket$ whose entries lie in $S$ such that $(z\id_n-\varphi,z\id_n+\varphi)$ is a matrix factorization for $z^2+g$. If $M=\cok(\overline{z\id_n-\varphi})$, then the trace ideal of $M$ is generated by the entries of $\overline{z\id_n+\varphi}$.

\end{corollary}

\proof By the previous theorem, the image of the injective map \[\Hom_R(M,R)\hookrightarrow\Hom_R(R^n,R)\] is isomorphic to $\im(\overline{z\id_n+\varphi^\intercal})$ and thus $\Hom_R(M,R)\cong\im(\overline{z\id_n+\varphi^\intercal})$. The generators of $\textim(\overline{z\id_n+\varphi^\intercal})$ are the columns of $\overline{z\id_n+\varphi^\intercal}$ and the corresponding generators of $\Hom_R(M,R)$ under the isomorphism are the maps $f_i:M\to R^n$  for $1\leq i\leq n$ defined by $f_i(e_j)=(\overline{z\id_n+\varphi^\intercal})_{ji}=(\overline{z\id_n+\varphi})_{ij}$. Then, $f_i(M)$ is generated by the entries of the $i$th row of $\overline{z\id_n+\varphi}$. Since $\Hom_R(M,R)$ is generated by the maps $f_1,f_2,\ldots,f_n$, the trace ideal is generated by the image of these maps, and hence, the trace ideal is generated by the entries of $\overline{z\id_n+\varphi}$. \qed

\begin{rem}
In the case that $\varphi$ has $0$'s as its diagonal entries, the ideal generated by the entries of $\overline{z\id_n+\varphi}$ is the same as the ideal generated by the entries of $\overline{z\id_n-\varphi}$.
\end{rem}

\section{Dimension One ADE Singularities}
\label{sec:dim1ade}

\begin{rem}
The following two examples are not integral domains. For this reason, while we can compute $\tau_{MCM}(R)$, it may not give us helpful information about the singularities of the ring. This will be apparent for Proposition \ref{prop:zero}.
\end{rem}

\begin{prop} \label{prop:zero} ($D_\infty$)
Let $R=k \llbracket x,y \rrbracket /(x^2y)$ where $k$ is a field of some arbitrary characteristic. Then $\tau_{MCM}(R)=(0)$.
\end{prop} 

\begin{proof} 

Two of the indecomposable Cohen-Macaulay $R$-modules are the cokernels of the $1\times 1$ matrices $(y)$ and $(x^2)$ \cite[Example 14.23]{leuschkewiegand}. 

Let \[M_1=\cok(y)=R/(y)\] \[M_2=\cok(x^2)=R/(x^2).\] 

We claim that $\tau_{M_1}(R)=(x^2)$. 
First, the map $1\mapsto x^2$ is a well-defined element of $\Hom(M_1,R)$ since $1 \mapsto x^2$ is a homomorphism from $R \to R$ such that $y\mapsto x^2y=0$. Now, let $f\in\Hom(M_1,R)$. Then $yf(1)=f(y)=f(0)=0$ which implies $f(1)\in\ann_R(y)=(x^2)R$. Thus, $1 \mapsto x^2$ generates $\Hom(M_1,R$) and therefore, $\tau_{M_1}=(x^2)$.

Next, we claim that $\tau_{M_2}(R)=(y)$. The map $1\mapsto y$ is a well-defined element of $\Hom(M_2,R)$ since the map gives $x^2\mapsto yx^2=0$. Now, let $f\in\Hom(M_2,R)$. Then $x^2f(1)=f(x^2)=f(0)=0$ which implies $f(1)\in\ann_R(x^2)=(y)R$. Thus, $1 \mapsto y$ generates $\Hom(M_2,R$) and therefore, $\tau_{M_2}=(y)$.

Thus, we see \[\tau_{MCM}(R)\subset (x^2)\cap (y)=(x^2y)=(0). \qedhere\] \end{proof}

In consequence, the singular test ideal $\tau_{sing}(R)$ (obtained by intersecting $\tau_M(R)$ for \textit{all} big \CM\ $R$-modules $M$, see \cite{PRC}) is equal to zero in this example as well. This result very much justifies the hypothesis that $R$ is a domain in the results of \cite{PRC} that describe how $\tau_{sing}(R)$ is connected to the singularities of the ring.

In proving the next result, we see our first application of Corollary \ref{corollary:factorization}.

\begin{prop} 
Let R be the ring $R=k\llbracket x,y \rrbracket /(y^2)$ where $k$ is a perfect field. Then $\tau_{MCM}(R)=(y)$.

\end{prop}

\begin{proof}
Let $I_n=(x^n,y)$ for $0 \le n < \infty$, and $I_{\infty}=(y)$. Up to isomorphism, the set of indecomposable maximal Cohen-Macaulay $R$-modules is
\[\cM=\{I_n : n\in\N\cup\{\infty\}\}\]
\cite[Example 6.5]{yoshino}. We claim that for each $n\in\mathbb{N}$, $\tau_{(x^n,y)}(R)=(x^n,y)$ and $\tau_{(y)}(R)=(y)$. 
For each $n$, let

\[\varphi_n=
\begin{pmatrix}
0&0\\
x^n&0
\end{pmatrix}
\]

and let $M_n=\coker(y\id_2-\varphi_n)$. Similarly, let $\varphi_\infty$ be the $1\times 1$ matrix $(0)$ and $M_\infty=\coker(y\id_1-\varphi_\infty)=R/(y)$. Observe that the $R$-linear map $R^2\to (x^n,y)$ which sends $\begin{psmallmatrix}
1\\ 0
\end{psmallmatrix}$ to $x^n$ and $\begin{psmallmatrix}
0\\ 1
\end{psmallmatrix}$ to $y$ is a surjective map whose kernel is the image of $y\id_2-\varphi_n$ and therefore $M_n\cong(x^ n,y)=I_n$ as $R$-modules. Similarly, $M_\infty\cong (y)=I_\infty$. Furthermore, we see $\varphi_n^2=0$ and $\varphi_\infty^2=0$ which implies that $(y\id-\varphi_n,y\id+\varphi_n)$ is a matrix factorization for $y^2$ for all $n\in\mathbb{N}\cup\{\infty\}$. So, by Corollary \ref{corollary:factorization}, the trace ideal of $I_n$ is generated by the entries of $y\id+\varphi_n$. Thus, $\tau_{I_n}=(x^n,y)$ for all $n\in\mathbb{N}$ and $\tau_{I_\infty}(R)=(y)$ and hence $\tau_{MCM}(R)=(y)$.
\end{proof}

\begin{prop} 
Let $R=k\llbracket x,y \rrbracket /(y^2+x^n)$ where $k$ is an algebraically closed field and $n$ is an odd positive integer.  Then for each $j\in \{1,2,...,\frac{n-1}{2}\}$, $\tau_{(x^j,y)}(R)=(x^j,y)$ and thus $\tau_{MCM}(R)=(x^{\frac{n-1}{2}},y)$.
\end{prop}

\begin{proof}
Up to isomorphism, the set of indecomposable maximal Cohen-Macaulay modules is $\cM=\{R,(x,y),(x^2,y),...,(x^{\frac{n-1}{2}},y)\}$ \cite[Proposition 5.11]{yoshino}. For $j\in\{0,1,\ldots,\frac{n-1}{2}\}$, let

\[\varphi_j=
\begin{pmatrix}
0&-x^{n-j}\\
x^j&0
\end{pmatrix}
\]

and let $M_j=\coker(y\id_2-\varphi_j)$. Observe that the $R$-linear map $R^2\to (x^j,y)$ which sends $\begin{psmallmatrix}
1\\ 0
\end{psmallmatrix}$ to $x^j$ and $\begin{psmallmatrix}
0\\ 1
\end{psmallmatrix}$ to $y$ is a surjective map whose kernel is the image of $y\id_2-\varphi_j$ and therefore $M_j\cong(x^j,y)$ as $R$-modules. Furthermore,
\[\varphi_j^2=\begin{pmatrix}
-x^n&0\\
0&-x^n
\end{pmatrix}=-x^n\id_2,\]
which implies that $(y\id_2-\varphi_j,y\id_2+\varphi_j)$ is a matrix factorization for $y^2+x^n$. So, by Corollary \ref{corollary:factorization}, the trace ideal of $(x^j,y)$ is generated by the entries of $y\id_2+\varphi$. Since $j\leq\frac{n-1}{2}$ (and therefore $j< n-j$) we conclude $\tau_{(x^j,y)}(R)=(x^j,y)$, and hence that $\tau_{MCM}(R)=(x^{\frac{n-1}{2}},y)$. 
\end{proof}

\begin{rem}
Note that $\tau_{MCM}(R)$ is the conductor ideal of $R$, i.e., the ideal consisting of elements of the integral closure $\bar{R}$ of $R$ in its fraction field that multiply $\bar{R}$ back into $R$. By \cite[Proposition 4.18]{PRC}, the conductor is a lower bound for $\tau_{MCM}(R)$. So in this example, $\tau_{MCM}(R)$ is as small as possible.
\end{rem}

\section{Dimension Two ADE Singularities}
\label{sec:dim2ade}

We use the results from Section \ref{sec:maintheorem} to compute the \MCM\ module test ideals of 
%examples whose indecomposable $MCM$ modules over 
dimension 2 ADE singularities (also known as Kleinian singularities). The indecomposable \MCM\ modules of these rings are cokernels of square matrix factorizations of the form  $(zI-\varphi)(zI+\varphi)=(z^2+g(x,y))I$. The computations for the test ideals of Kleinian singularities $E_6$, $E_7$, and $E_8$ are simplified significantly by Corollary \ref{corollary:factorization}.

\begin{rem}
The ring in Proposition \ref{prop:zero2} is not a domain.  Similar to Proposition \ref{prop:zero}, while we can compute $\tau_{MCM}(R)$, it does not give us useful information on the singularities of the ring.  
\end{rem}

\begin{prop}
\label{prop:zero2}
Let $R=k\llbracket x,y,z\rrbracket/(x^2+z^2)$ where $k$ is an algebraically closed field of characteristic not equal to 2. Then $\tau_{MCM}(R)=(0)$.
\end{prop} 

\begin{proof}

Let $i \in k$ denote $\sqrt{-1}$. Then the indecomposable non-free \MCM\ $R$-modules are isomorphic to $\cok(z\id-\phi)$ where $\phi$ is one of the following matrices over $R$:\\

\begin{itemize}
\begin{multicols}{2}
    \item[]$ \varphi_0=(ix)$ or $\varphi_0'=(-ix)$; or
    \item[]$\varphi_j=\begin{pmatrix}
    -ix       & y^j \\
    0       & ix \\
\end{pmatrix}$ for some $j \geq 1$ 
\end{multicols}
\end{itemize}

\noindent\cite[14.17. Proposition]{leuschkewiegand}. 

By Corollary $\ref{corollary:factorization}$, $\tau_{\cok(z\id_1-\varphi_0)}=(z+ix)$ and $\tau_{\cok(z\id_1-\varphi_0')}=(z-ix)$. However, $(z+ix)\cap(z-ix)=(z^2+x^2)=(0)$ and thus $\tau_{MCM}=(0)$.
\end{proof}

\begin{rem}
The following Proposition duplicates Example 5.5 of \cite{PRC}, but using Theorem \ref{thm:maintheorem} shortens the proof significantly.
\end{rem}

\begin{prop} 
Let $R=k\llbracket x,y,z\rrbracket/(x^2y+z^2)$, where $k$ is a field of some arbitrary characteristic. 
Then $\tau_{MCM}(R)=(x^2,z)$. %where the MCM is the $\cok(z\id_n-\phi)$ for the matrices shown below.
\end{prop}

\begin{proof}
 The indecomposable non-free \MCM\ $R$-modules are isomorphic to $\cok(z\id-\phi)$ where $\phi$ is one of the following matrices over $k\llbracket x,y\rrbracket$: 

\begin{multicols}{2}
\begin{itemize}
 \item[ ]
$\varphi_1 = \begin{pmatrix}
    0       & -y \\
    x^2       & 0 \\
\end{pmatrix}$

\item[ ]
$ \varphi_{3,j} = \begin{pmatrix}
     &   & -xy  & 0 \\
     &   & -y^{j+1}  & xy \\
    x & 0  &   &  \\
    y^j & -x &  &  \\
    
\end{pmatrix}$

\end{itemize}
\end{multicols}

\begin{multicols}{2}
\begin{itemize}
 \item[]
$ \varphi_{2} = \begin{pmatrix}
    0       & -xy \\
    x       & 0 \\
\end{pmatrix}$

\item[]
$\varphi_{4,j} =\begin{pmatrix}
     &   & -xy  & 0 \\
     &   & -y^{j}  & x \\
    x & 0  &   &  \\
    y^j & -xy &  &  \\
    
\end{pmatrix}$

\end{itemize}

\end{multicols}

%example citation
where $j \ge 1$ \cite[14.19. Proposition]{leuschkewiegand}. For each $\varphi$ in the above list, $(z\id-\varphi,z\id+\varphi)$ is a matrix factorization for $z^2+x^2y$ over $k\llbracket x,y,z \rrbracket$. By Corollary \ref{corollary:factorization},  the trace ideal of $\cok(z\id-\varphi)$ is generated by the entries of $z\id+\varphi$ where $\varphi$ is any of the matrices in the above list. Thus,  $\tau_{MCM}(R)=(x^2,z)$. 
\end{proof}

\begin{rem}
Note that $\tau_{MCM}(R)$ is actually smaller than the conductor ideal in this case: the conductor is $(x,z)$ \cite[Proposition 14.19]{leuschkewiegand} but $\tau_{MCM}(R)=(x^2,z)$. This indicates that \cite[Proposition 4.18]{PRC} does not extend to higher dimensional rings.
\end{rem}

\begin{rem}[Assumption]
We will assume for the remaining examples in this section that the base field is algebraically closed and not of characteristic 2, 3, or 5. This is so that we can use Theorem 6.23 of \cite{leuschkewiegand}, which proves that these rings have finite \CM\ type under that hypothesis.
\end{rem}

\begin{prop} (Two-Dimensional $A_n$)
Let $R=k\llbracket x,y,z\rrbracket /(z^2+x^2 +y^{n+1})$ where $k$ is a field (which is algebraically closed and not of characteristic not equal to $2$, $3$, or $5$). Then $\tau_{MCM}(R)=(x,y^{\lfloor \frac{n+1}{2}\rfloor})$.
\end{prop}

\begin{proof}
Up to isomorphism, the non-free indecomposable \MCM\ $R$-modules are $M_j=\coker(z\id_2-\varphi_j)$ where 

\[\varphi_j =\begin{pmatrix}
    ix & y^{n+1-j}  \\
    -y^{j} & -ix \\
\end{pmatrix}\]

for $j \in  \{0,...,n\}$ \cite[9.20]{leuschkewiegand}. So by Corollary $\ref{corollary:factorization}$ the trace ideal of $M_j=\coker(z\id_2-\varphi_j)$ is generated by the entries of $z\id_2-\varphi_j$. Thus, $$\tau_{M_j}(R)=(x,y^{\min\{j,n+1-j\}}).$$ Consequently, the intersection is $\tau_{MCM}(R)=(x,y^{\lfloor\frac{n+1}{2}\rfloor})$.
\qedhere
\end{proof}

\begin{prop}  (Two-Dimensional $D_n$)

Let $k$ be a field and let $R=k\llbracket x,y,z\rrbracket/(z^2+x^2y+y^{n-1})$ for some $n\geq 4$. The \MCM\ module test ideal of $R$ is
%intersection of the test ideals of $R$ is
$\tau_{MCM}(R)=(x^2,y^{\lfloor n/2\rfloor},z)$.
\end{prop}

\begin{proof}

The indecomposable \MCM\ modules of $R$ are given by $M_j=\cok(z\id-\varphi_j)$ where $\phi_j$ is the following:
 
For $j=1$, let

\[\varphi_1 =\begin{pmatrix}
    0 & -x^2-y^{n-2}  \\
    y & 0 \\
\end{pmatrix}.\]

For $j\in\{2,3,\ldots,n-2\}$, if $j$ is even, let

\[\varphi_j =\begin{pmatrix}
     0 & 0 & -xy & -y^{n-1-j/2} \\
    0 & 0 & -y^{j/2} & x \\
    x & y^{n - 1 - j / 2} & 0 & 0 \\
    y^{j/2} & -xy & 0 & 0\\
\end{pmatrix}\]

and if $j$ is odd, let

\[\varphi_j =\begin{pmatrix}
     0 & 0 & -xy & -y^{n-1-(j-1)/2} \\
    0 & 0 & -y^{(j+1)/2} & xy \\
    x & y^{n-2-(j-1) / 2} & 0 & 0 \\
    y^{(j-1)/2} & -x & 0 & 0\\
\end{pmatrix}.\]

If $n$ is odd, let

\[\begin{aligned} \varphi_{n-1} &= \begin{pmatrix}
    iy^{(n-1)/2} & -x  \\
    xy & -iy^{(n-1)/2} \\
\end{pmatrix}\\ \varphi_n &= \begin{pmatrix}
    iy^{(n-1)/2} & -xy  \\
    x & -iy^{(n-1)/2} \\
\end{pmatrix}\\ \end{aligned}\]

and if $n$ is even, let

\[\begin{aligned}\varphi_{n-1} &=\cok \begin{pmatrix}
    0 & -x-iy^{(n-2)/2}  \\
    xy-iy^{n/2} & 0 \\
\end{pmatrix}\\ \varphi_{n} &=\cok \begin{pmatrix}
    0 & -x+iy^{(n-2)/2}  \\
    xy+iy^{n/2} & 0 \\
\end{pmatrix}\\
\end{aligned}\]
\cite[9.21]{leuschkewiegand}.
 
 So by Corollary \ref{corollary:factorization} the trace ideal of $M_j=\cok(z\id-\varphi_j)$ is generated by the entries of $z\id+\varphi_j$. Consequently, 
 \[\tau_{M_1}(R)=(x^2,y,z).\]
 Given $j\in\{2,3,\ldots,n-2\}$, we have $j<n-1$ so that $j/2<n-1-j/2$ and thus
 \[\tau_{M_j}(R)=(x,y^{j/2},z)\]
 for even values of $j$ and
 \[\tau_{M_j}(R)=(x,y^{(j-1)/2},z)\]
 for odd values of $j$.

If $n$ is odd, we see
\[\begin{aligned}
\tau_{M_{n-1}}(R)&=(x,z+iy^{(n-1)/2},z-iy^{(n-1)/2})=(x,y^{(n-1)/2},z),\\
\tau_{M_n}(R)&=(x,z-iy^{(n-1)/2},z+iy^{(n-1)/2})=(x,y^{(n-1)/2},z).
\end{aligned}\]
If $n$ is even, we see
\[\begin{aligned}
\tau_{M_{n-1}}(R)&=(x+iy^{(n-2)/2},xy-iy^{n/2},z),\\
\tau_{M_n}(R)&=(x-iy^{(n-2)/2},xy+iy^{n/2},z).
\end{aligned}\]

Hence if $n$ is odd,  $\tau_{MCM}(R)=(x^2,y^{(n-1)/2},z)$. 

If $n$ is even, we first claim that
\[\tau_{M_{n-1}}(R)\cap\tau_{M_n}(R)=(x^2,xy,y^{n/2},z).\]
Indeed,
\[\begin{aligned}
    x^2&=(x+\frac{1}{2i}y^{(n-2)/2})(x+iy^{(n-2)/2})+\frac{1}{2i}y^{(n-4)/2}(xy-iy^{n/2})\in\tau_{M_{n-1}}(R)\\
    xy&=\frac{1}{2}y(x+iy^{(n-2)/2})+\frac{1}{2}(xy-iy^{n/2})\in\tau_{M_{n-1}}(R)\\
    y^{n/2}&=\frac{1}{2i}y(x+iy^{(n-2)/2})-\frac{1}{2i}(xy-iy^{n/2})\in\tau_{M_{n-1}}(R)\\
    x^2&=(x-\frac{1}{2i}y^{(n-2)/2})(x-iy^{(n-2)/2})-\frac{1}{2i}y^{(n-4)/2}(xy+iy^{n/2})\in\tau_{M_n}(R)\\
    xy&=\frac{1}{2}y(x-iy^{(n-2)/2})+\frac{1}{2}(xy+iy^{n/2})\in\tau_{M_n}(R)\\
    y^{n/2}&=-\frac{1}{2i}y(x-iy^{(n-2)/2})+\frac{1}{2i}(xy+iy^{n/2})\in\tau_{M_n}(R)
\end{aligned}\]
so that \[\tau_{M_{n-1}}(R)\cap\tau_{M_n}(R)\supset(x^2,xy,y^{n/2},z).\] Conversely, any element $p$ of \[(\tau_{M_{n-1}}(R)\cap\tau_{M_n}(R))/(x^2,xy,y^{n/2},z)\] must simultaneously be of the form $q_1(x+iy^{(n-2)/2})$ and $q_2(x-iy^{(n-2)/2})$ for $q_1,q_2\in R/(x^2,xy,y^{n/2},z)$. However, because $x$, $y$, and $z$ annihilate $x\pm iy^{(n-2)/2}$, we may suppose that $q_1$ and $q_2$ are constants. However, \[q_1(x+iy^{(n-2)/2})=q_2(x-iy^{(n-2)/2})\] implies 
\[(q_2-q_1)x=(q_2+q_1)iy^{(n-2)/2}.\] Since $q_1$ and $q_2$ are constants, this is only possible if $q_2-q_1=0$ and $q_2+q_1=0$ so that $q_1=q_2=0$ and thus $p=0$. Hence, $\tau_{M_{n-1}}(R)\cap\tau_{M_n}(R)=(x^2,xy,y^{n/2},z)$ as claimed. Since $\tau_{M_j}(R)\supset(x^2,y^{n/2},z)$ for $j<n-1$, we conclude $\tau_{MCM}(R)=(x^2,y^{n/2},z)$ when $n$ is even as well.
\end{proof}

\begin{prop} (Two-Dimensional $E_6$)
Let $k$ be a field and let $R=k\llbracket x,y,z\rrbracket/(z^2+x^3+y^3)$. Then  $\tau_{MCM}(R)=(x,y^2,z)$. 
\end{prop}

\begin{proof}
The indecomposable \MCM\ $R$-modules are given by $\cok(z\id-\varphi)$, where $\phi$ is one of the matrices below \cite[9.22]{leuschkewiegand}.
\[\begin{aligned}
\varphi_1 &= \begin{pmatrix}
      &  & -x^2 & -y^{2} \\
     &  & -y & x \\
    x & y^{3} &  &  \\
    y & -x^2 &  & \\
\end{pmatrix}  & \hfill \varphi_2 &= \begin{pmatrix}
     & & & -x^2 & -y^3 & xy^{2} \\
     & & & xy & -x^2 & -y^3 \\
     & & & -y^2 & xy & -x^2 \\
    x & 0 & y^2 & & & \\
    y & x & 0 & & & \\
    0 & y & x & & & \\
\end{pmatrix}\\
\varphi_3 &= \begin{pmatrix}
     iy^2 & 0 & -x^2 & 0 \\
     0 & iy^2 & -xy & -x^2 \\
    x & 0 & -iy^2 & 0 \\
    -y & x & 0 & -iy^2\\
\end{pmatrix}  &  \varphi_4 &= \begin{pmatrix}
     -iy^2 & -x^2 \\
     x & iy^2 \\
\end{pmatrix}
\end{aligned}\]
For each $j$, $(z\id-\varphi_j,z\id+\varphi_j)$ is a matrix factorization of $z^2+x^3+y^3$ over $k\llbracket x,y,z \rrbracket$. 
By Corollary \ref{corollary:factorization}, the trace ideal of $\cok(z\id-\varphi_j)$ is generated by the entries of $\varphi_j$. Thus,  $\tau_{MCM}(R)=(x,y^2,z)$. 
\end{proof}

\begin{prop}
(Two-Dimensional $E_7$)
Let $k$ be a field and let $R=k\llbracket x,y,z\rrbracket/(z^2+x^3+xy^3)$. Then  $\tau_{MCM}(R)=(x,y^3,z)$.
\end{prop}

\begin{proof}
The indecomposable \MCM\ $R$-modules are given by $\coker(z\id-\varphi)$, where $\phi$ is one of the following matrices \cite[9.23]{leuschkewiegand}:
\[
\varphi_1 = \begin{pmatrix}
      &  & -x^2 & -xy^{2} \\
     &  & -y & x \\
    x & xy^{2} &  &  \\
    y & -x^2 &  & \\
\end{pmatrix}     \varphi_2 = \begin{pmatrix}
     & & & -x^2 & -xy^2 & x^2y \\
     & & & xy & -x^2 & -xy^3 \\
     & & & -y^2 & xy & -x^2 \\
    x & 0 & xy & & & \\
    y & x & 0 & & & \\
    0 & y & x & & & \\
\end{pmatrix}\]
\[\varphi_3 = \begin{pmatrix}
     & & & & 0 & 0 & -x^2 & -xy^2 \\
     & & & & 0 & 0 & -xy & -x^2 \\
     & & & & -x & -y^2 & 0 & -xy \\
     & & & & -y & x & -x & 0 \\
    0 & -xy & x^2 & xy^2 & & & & \\
    x & 0 & xy & -x^2 & & & & \\
    x & y^2 & 0 & 0 & & & & \\
    y & -x & 0 & 0 & & & & \\
\end{pmatrix}
\]
\[\begin{aligned}
\varphi_4 &= \begin{pmatrix}
     & & & xy & -x^2 & -xy^2 \\
     & & & -y^2 & xy & -x^2 \\
     & & & -x & -y^2 & xy \\
    0 & xy & x^2 & & & \\
    x & 0 & xy & & & \\
    y & x & 0 & & & \\
\end{pmatrix} & \hfill \varphi_5 &= \begin{pmatrix}
      &  & -xy & -x^{2} \\
     &  & -x & y^2 \\
    y^2 & x^{2} &  &  \\
    x & -xy &  & \\
\end{pmatrix}\\
\varphi_6 & = \begin{pmatrix}
     0 & y^3 + x^{2} \\
     -x &  0 \\
\end{pmatrix}          & \hfill \varphi_7 &= \begin{pmatrix}
      &  & -x^2 & -xy^{2} \\
     &  & -xy & x^2 \\
    x & y^{2} &  &  \\
    y & -x &  & \\
\end{pmatrix}
\end{aligned}\]

  For each $j$, $(z\id-\varphi_j,z\id+\varphi_j)$ is a matrix factorization of $z^2+x^3+xy^3$ over $k\llbracket x,y,z \rrbracket$. So by Corollary \ref{corollary:factorization} the trace ideal of $\cok(z\id-\varphi_j)$ is generated by the entries of $\varphi_j$. Thus,  $\tau_{MCM}(R)=(x,y^3,z)$. 
\end{proof}

\begin{prop}  (Two-Dimensional $E_8$)

Let $k$ be a field and let $R=\llbracket x,y,z\rrbracket/(z^2+x^3+y^5)$. Then  $\tau_{MCM}(R)=(x,y^2,z)$.\end{prop}

\begin{proof}

 The indecomposable MCM $R$-modules are given by $\coker(z\id-\varphi)$ where $\varphi$ is one of the following matrices \cite[9.24]{leuschkewiegand}:
\[
\varphi_1 = \begin{pmatrix}
      &  & -x^2 & -y^{4} \\
     &  & -y & x \\
    x & y^{4} &  &  \\
    y & -x^2 &  & \\
\end{pmatrix}
\quad
\varphi_2 = \begin{pmatrix}
     & & & -x^2 & -y^4 & xy^3 \\
     & & & xy & -x^2 & -y^4 \\
     & & & -y^2 & xy & -x^2 \\
    x & 0 & y^3 & & & \\
    y & x & 0 & & & \\
    0 & y & x & & & \\
\end{pmatrix}\]
\[
\varphi_3 = \begin{pmatrix}
     & & & & xy & -y^2 & -x^2 & 0 \\
     & & & & -y^3 & 0 & 0 & -x \\
     & & & & x^2 & 0 & 0 & -y^2 \\
     & & & & 0 & x & -y^3 & -y \\
    0 & y^2 & -x & 0 & & & & \\
    y^3 & xy & 0 & -x^2 & & & & \\
    x & 0 & -y & y^2 & & & & \\
    0 & x^2 & y^3 & 0 & & & & \\
\end{pmatrix}\\
\]

\[\begin{aligned}
\varphi_4 &= \begin{pmatrix}
     & & & & & -y^3 & x^2 & 0 & 0 & 0 \\
     & & & & & 0 & y^3 & -x^2 & xy^2 & -y^4 \\
     & & & & & 0 & -xy & -y^3 & -x^2 & xy^2 \\
     & & & & & y^2 & 0 & xy & -y^3 & -x^2 \\
     & & & & & -x & -y^2 & 0 & 0 & 0 \\
     y^2 & 0 & 0 & 0 & x^2 & & & & & \\
     -x & 0 & 0 & 0 & y^3 & & & & & \\
     0 & x & y^2 & 0 & 0 & & & & & \\
     y & 0 & x & y^2 & 0 & & & & & \\
     0 & y & 0 & x & y^2 & & & & & \\
\end{pmatrix}\\
\varphi_5 &= \begin{pmatrix}
     & & & & & & 0 & 0 & 0 & -x^2 & xy^2 & -y^4\\
     & & & & & & 0 & 0 & 0 & -y^3 & -x^2 & xy^2\\
     & & & & & & 0 & 0 & 0 & xy & -y^3 & -x^2\\
     & & & & & & -x & -y^2 & 0 & 0 & 0 & y^3\\
     & & & & & & 0 & -x & -y^2 & y^2 & 0 & 0\\
     & & & & & & -y & 0 & -x & 0 & y^2 & 0\\
     0 & 0 & y^3 & x^2 & -xy^2 & y^4 & & & & & & \\
     y^2 & 0 & 0 & y^3 & x^2 & -xy^2 & & & & & & \\
     0 & y^2 & 0 & -xy & y^3 & x^2 & & & & & & \\
     x & y^2 & 0 & 0 & 0 & 0 & & & & & & \\
     0 & x & y^2 & 0 & 0 & 0 & & & & & & \\
     y & 0 & x & 0 & 0 & 0 & & & & & & \\
\end{pmatrix}\\
\varphi_6 &= \begin{pmatrix}
     & & & & 0 & -y^3 & -x^2 & 0 \\
     & & & & -y^2 & 0 & xy & -x^2 \\
     & & & & -x & -y^2 & 0 & y^3 \\
     & & & & 0 & -x & y^2 & 0 \\
    0 & y^3 & x^2 & -xy^2 & & & & \\
    y^2 & 0 & 0 & x^2 & & & & \\
    x & 0 & 0 & -y^3 & & & & \\
    y & x & -y^2 & 0 & & & & \\
\end{pmatrix}\\
\end{aligned}\]
\[
\varphi_7 = \begin{pmatrix}
      &  & -y^3 & -x^{2} \\
     &  & x & -y^2 \\
    y^2 & -x^{2} &  &  \\
    x & y^3 &  & \\
\end{pmatrix} \quad \quad
\varphi_8 = \begin{pmatrix}
     & & & -x^2 & -xy^2 & -y^4 \\
     & & & -y^3 & -x^2 & xy^2 \\
     & & & xy & -y^3 & -x^2 \\
    x & y^2 & 0 & & & \\
    0 & x & y^2 & & & \\
    y & 0 & x & & & \\
\end{pmatrix}
\]

For each $j$, $(z\id-\varphi_j,z\id+\varphi_j)$ is a matrix factorization of $z^2+x^3+y^5$ over $k\llbracket x,y,z \rrbracket$. So by Corollary \ref{corollary:factorization} the trace ideal of $\cok(z\id-\varphi_j)$ is generated by the entries of $\varphi_j$. Thus,  $\tau_{MCM}(R)=(x,y^2,z)$. \end{proof}

\section*{Acknowledgments}

We would like to thank Eleonore Faber, Graham Leuschke, and Haydee Lindo for helpful conversations during the writing of this paper. We would also like to thank the Mason Experimental Geometry Lab (\url{http://meglab.wikidot.com/}), which sponsored this research group, and in particular its
director Sean Lawton and co-director Anton Lukyanenko.

\bibliographystyle{amsalpha}
\bibliography{mainbib}
\end{document}